\documentclass[12 pt]{amsart}
\usepackage{amsmath}
\usepackage{amssymb}
\usepackage{enumerate}
\usepackage{comment}
\usepackage{hyperref}
\usepackage[letterpaper, top = 3.3cm, bottom = 3.3cm,right = 3.1cm, left = 3.1cm]{geometry}
\numberwithin{equation}{section}

\numberwithin{equation}{section}

\theoremstyle{theorem}
\newtheorem{theorem}{Theorem}[section]
\newtheorem{lemma}[theorem]{Lemma}
\newtheorem{proposition}[theorem]{Proposition}
\newtheorem{corollary}[theorem]{Corollary}

\theoremstyle{definition}
\newtheorem{definition}[theorem]{Definition}

\theoremstyle{remark}
\newtheorem*{remark}{Remark}

\numberwithin{equation}{section}
\begin{document}

\title{A note on species realizations and nondegeneracy of potentials}

\author{Daniel L\'{o}pez-Aguayo}
 \address{Tecnologico de Monterrey, Escuela de Ingenier\'{i}a y Ciencias, Hidalgo, M\'{e}xico.}
\email{dlopez.aguayo@itesm.mx}

\maketitle

\begin{abstract} In this note we show that a mutation theory of species with potential can be defined so that a certain class of skew-symmetrizable integer matrices have a species realization admitting a non-degenerate potential. This gives a partial affirmative answer to a question raised by Jan Geuenich and Daniel Labardini-Fragoso. We also provide an example of a class of skew-symmetrizable $4 \times 4$ integer matrices, which are not globally unfoldable nor strongly primitive, and that have a species realization admitting a non-degenerate potential.
\end{abstract}

\section{Introduction}  In \cite[p.14]{6}, motivated by the seminal paper \cite{5} of Derksen-Weyman-Zelevinsky, J. Geuenich and D. Labardini-Fragoso raise the following question: \\

\textbf{Question} \cite[Question 2.23]{6} Can a mutation theory of species with potential be defined so that every skew-symmetrizable matrix $B$ have a species realization which admit a non-degenerate potential? \\

In \cite{1}, we show that for every skew-symmetrizable matrix $B=(b_{ij})_{i,j} \in \mathbb{Z}^{n \times n}$ that admits a skew-symmetrizer $D=\operatorname{diag}(d_{1},\hdots, d_{n})$, with the property that each $d_{j}$ divides each and every element $b_{ij}$, then for every finite sequence $k_{1}, \hdots, k_{l}$ of elements of $\{1,\hdots, n\}$, there exists a species realization of $B$ and a potential $P$ on this species such that all the iterated mutations $\bar{\mu}_{k_{1}}P$, $\bar{\mu}_{k_{2}}\bar{\mu}_{k_{1}}P, \hdots, \bar{\mu}_{k_{l}} \cdots \bar{\mu}_{k_{1}}P$ are $2$-acyclic. 

In \cite{7}, D. Labardini-Fragoso and A. Zelevinsky give a partially positive answer to Question $2.23$ provided that the skew-symmetrizer has pairwise coprime diagonal entries. We remark that this is a stronger condition than the one we impose in \cite{1}. Indeed, since $DB$ is skew-symmetric then $-d_{j}b_{ji}=d_{i}b_{ij}$. Using the fact that $b_{ji}$ is an integer and that $(d_{i},d_{j})=1$, it follows that $d_{j}$ divides $b_{ij}$, as claimed. 

In Theorem \ref{theo1} and Corollary \ref{cor} we give a partially affirmative answer to Question $2.23$ by proving the following: let $B=(b_{ij}) \in \mathbb{Z}^{n \times n}$ be a skew-symmetrizable matrix with skew-symmetrizer $D=\operatorname{diag}(d_{1},\ldots,d_{n})$. If $d_{j}$ divides $b_{ij}$ for every $j$ and every $i$, then the matrix $B$ can be realized by a species that admits a non-degenerate potential.

Finally, in Section 5 we give an example of a class of skew-symmetrizable $4 \times 4$ integer matrices which are not globally unfoldable nor strongly primitive (in the sense of \cite[Definition 14.1]{7}), and that have a species realization admitting a non-degenerate potential. This gives an example of a class of skew-symmetrizable $4 \times 4$ integer matrices which are not covered by \cite{7}. 

\section{Preliminaries}
The following material is taken from \cite{1}. 
\begin{definition} Let $F$ be a field and let $D_{1},\hdots,D_{n}$ be division rings, each containing $F$ in its center and of finite dimension over $F$. Let  $S=\displaystyle \prod_{i=1}^{n} D_{i}$ and let $M$ be an $S$-bimodule of finite dimension over $F$. Define the \emph{algebra of formal power series} over $M$ as the set: 
\begin{center}
$\mathcal{F}_{S}(M):=\left\{\displaystyle \sum_{i=0}^{\infty} a(i): a(i) \in M^{\otimes i}\right\}$
\end{center}
\end{definition} 
where $M^{0}=S$. 
Note that $\mathcal{F}_{S}(M)$ is an associative unital $F$-algebra where the product is the one obtained by extending the product of the tensor algebra $T_{S}(M)=\displaystyle \bigoplus_{i=0}^{\infty} M^{\otimes i}$.

Let $\{e_{1},\ldots,e_{n}\}$ be a complete set of primitive orthogonal idempotents of $S$.

\begin{definition} An element $m \in M$ is \emph{legible} if $m=e_{i}me_{j}$ for some idempotents $e_{i},e_{j}$ of $S$.
\end{definition}

\begin{definition} Let $Z=\displaystyle \sum_{i=1}^{n} Fe_{i} \subseteq S$. We say that $M$ is $Z$-\emph{freely generated} by a $Z$-subbimodule $M_{0}$ of $M$ if the map $\mu_{M}: S \otimes_{Z} M_{0} \otimes_{Z} S \rightarrow M$ given by $\mu_{M}(s_{1} \otimes m \otimes s_{2})=s_{1}ms_{2}$ is an isomorphism of $S$-bimodules. In this case we say that $M$ is an $S$-bimodule which is $Z$-freely generated.
\end{definition}

Throughout this paper we will assume that $M$ is $Z$-freely generated by $M_{0}$. 

\begin{definition}
Let $A$ be an associative unital $F$-algebra. A cyclic derivation, in the sense of Rota-Sagan-Stein \cite{9}, is an $F$-linear function $\mathfrak{h}:A\rightarrow \mathrm{End}_{F}(A)$ such that:
\begin{equation}
\mathfrak{h}(f_{1}f_{2})(f)=\mathfrak{h}(f_{1})(f_{2}f)+\mathfrak{h}(f_{2})(ff_{1})
\label{eq1}
\end{equation}
for all $f,f_{1},f_{2} \in A$.
\end{definition}

\begin{definition} Let $A$ be an associative unital $F$-algebra. Given a cyclic derivation $\mathfrak{h}: A\rightarrow \mathrm{End}_{F}(A)$, we define the associated cyclic derivative $\delta: A \rightarrow A$ as 
$\delta (f)=\mathfrak{h}(f)(1)$.
\end{definition}

From \eqref{eq1} one obtains:
\begin{equation}
\delta (f_{1}f_{2})=\mathfrak{h}(f_{1})(f_{2})+\mathfrak{h}(f_{2})(f_{1})
\end{equation}
for all $f_{1}, f_{2}\in A$. In particular, $\delta (f_{1}f_{2})=\delta (f_{2}f_{1})$. \\

We now construct a cyclic derivative on $\mathcal{F}_{S}(M)$. First, we define a cyclic derivation on the tensor algebra $A=T_{S}(M)$ as follows. Consider the map
\begin{center}
$\hat{u}:A\times A\rightarrow A$
\end{center}
given by $\hat{u}(f,g)=\displaystyle \sum _{i=1}^{n}e_{i}gfe_{i}$ for every $f,g\in A$; this is an $F$-bilinear map which is $Z$-balanced. Therefore, there exists $u: A \otimes_{Z} A \rightarrow A$ such that $u(a \otimes b)=\hat{u}(a,b)$.
Now we define an $F$-derivation $\Delta: A\rightarrow A\otimes _{Z}A$ as follows: for $s\in S$, we define $\Delta (s)=1\otimes s-s\otimes 1$; for $m\in M_{0}$, we let $\Delta (m)=1\otimes m$. Then we define $\Delta: M\rightarrow  T_{S}(M)$ such that for $s_{1},s_{2}\in S$ and $m\in M_{0}$, we have
\begin{center}
$\Delta (s_{1}ms_{2})=\Delta (s_{1})ms_{2}+s_{1}\Delta (m)s_{2}+s_{1}m\Delta (s_{2}).$
\end{center}
The above map is well defined because $M\cong S\otimes _{Z}M_{0}\otimes _{Z}S$ via  the multiplication map $\mu_{M}$.
Now $\Delta $ can be extended to an  $F$-derivation on $A$.   

We define $\mathfrak{h}:A\rightarrow \mathrm{End}_{F}(A)$ as follows 
\begin{center}
$\mathfrak{h}(f)(g)=u(\Delta (f)g).$
\end{center}
We have:
\begin{align*}
\mathfrak{h}(f_{1}f_{2})(f)&=u(\Delta (f_{1}f_{2})f) \\
&=u(\Delta (f_{1})f_{2}f)+u(f_{1}\Delta (f_{2})f) \\
&=u(\Delta (f_{1})f_{2}f)+u(\Delta (f_{2})ff_{1}) \\
&=\mathfrak{h}(f_{1})(f_{2}f)+\mathfrak{h}(f_{2})(ff_{1}).
\end{align*}

Therefore $\mathfrak{h}$ is a cyclic derivation on $T_{S}(M)$. We now extend $\mathfrak{h}$ to $\mathcal{F}_{S}(M)$: take $f,g \in \mathcal{F}_{S}(M)$, then we have $\mathfrak{h}(f(i))(g(j))\in M^{\otimes i+j}$; thus we define $\mathfrak{h}(f)(g)(l)=\displaystyle \sum _{i+j=l}\mathfrak{h}(f(i))(g(j))$ for every non-negative integer $l$.
\begin{proposition} The $F$-linear map $\mathfrak{h}:\mathcal{F}_{S}(M)\rightarrow \mathrm{End}_{F}(\mathcal{F}_{S}(M))$ is a cyclic derivation.
\end{proposition}
\begin{proof} Let $f,f_{1},f_{2} \in \mathcal{F}_{S}(M)$. Then for any non-negative integer $l$, we have:
\begin{align*}
\mathfrak{h}(f_{1}f_{2})(f)(l)&=\sum _{i+j=l}\mathfrak{h}((f_{1}f_{2})(i))(f(j)) \\
&=\sum _{i_{1}+i_{2}+j=l}\mathfrak{h}(f_{1}(i_{1})f_{2}(i_{2}))(f(j)) \\
&=\sum _{i_{1}+i_{2}+j=l}\mathfrak{h}(f_{1}(i_{1}))(f_{2}(i_{2})f(j))+\sum _{i_{1}+i_{2}+j=l}\mathfrak{h}(f_{2}(i_{2}))(f(j)f_{1}(i_{1})) \\
&=\sum _{i_{1}+t=l}\mathfrak{h}(f_{1}(i_{1}))((f_{2}f)(t))+\sum _{i_{2}+r=l}\mathfrak{h}(f_{2}(i_{2}))((ff_{1})(r)) \\
&=\mathfrak{h}(f_{1})(f_{2}f)(l)+\mathfrak{h}(f_{2})(ff_{1})(l) \\
&=\left(\mathfrak{h}(f_{1})(f_{2}f)+\mathfrak{h}(f_{2})(ff_{1})\right)(l).
\end{align*}
The result follows.
\end{proof}
From the above we obtain a cyclic derivative $\delta$ on $\mathcal{F}_{S}(M)$ defined as
\begin{center}
$\delta (f):=\mathfrak{h}(f)(1)$
\end{center}
for every $f\in \mathcal{F}_{S}(M)$.

\begin{definition} Let $\mathcal{C}$ be a subset of $M$. We say that $\mathcal{C}$ is a right $S$-local basis of $M$ if every element of $\mathcal{C}$ is legible and if for each pair of idempotents $e_{i},e_{j}$ of $S$, we have that $\mathcal{C} \cap e_{i}Me_{j}$ is a $D_{j}$-basis for $e_{i}Me_{j}$.
\end{definition}

Note that a right $S$-local basis $\mathcal{C}$ induces a dual basis $\{u,u^{\ast}\}_{u \in \mathcal{C}}$, where $u^{\ast}: M_{S} \rightarrow S_{S}$ is the morphism of right $S$-modules defined by $u^{\ast}(v)=0$ if $v \in \mathcal{C} \setminus \{u\}$; and $u^{\ast}(u)=e_{j}$ if $u=e_{i}ue_{j}$. \\

Let $T$ be a $Z$-local basis of $M_{0}$ and $L$ be a $Z$-local basis of $S$. The former means that for each pair of idempotents $e_{i},e_{j}$ of $S$, $T \cap e_{i}Me_{j}$ is an $F$-basis of $e_{i}M_{0}e_{j}$, and the latter means that $L(i)=L \cap e_{i}S$ is an $F$-basis of the division algebra $e_{i}S=D_{i}$. It follows that the non-zero elements of the set $\{sa: s \in L, a \in T\}$ is a right $S$-local basis of $M$. Therefore, for every $s \in L$ and $aÊ\in T$, we have the map $(sa)^{\ast} \in \operatorname{Hom}_{S}(M_{S},S_{S})$ induced by the dual basis.

\begin{definition} Let $\mathcal{D}$ be a subset of $M$. We say that $\mathcal{D}$ is a left $S$-local basis of $M$ if every element of $\mathcal{D}$ is legible and if for each pair of idempotents $e_{i},e_{j}$ of $S$, we have that $\mathcal{D} \cap e_{i}Me_{j}$ is a $D_{i}$-basis for $e_{i}Me_{j}$. 
\end{definition}
As before, a left $S$-local basis $\mathcal{D}$ induces a dual basis $\{u,^{\ast}u\}_{u \in \mathcal{D}}$ where $^{\ast}u:  _{S}M \rightarrow _{S}S$ is the morphism of left $S$-modules defined by $^{\ast}u(v)=0$ if $v \in \mathcal{D} \setminus \{u\}$; and $^{\ast}u(u)=e_{i}$ if $u=e_{i}ue_{j}$. 

Let $\psi$ be any element of $\mathrm{Hom}_{S}(M_{S},S_{S})$. We will extend $\psi$ to an $F$-linear endomorphism of $\mathcal{F}_{S}(M)$, which we will denote by $\psi_{\ast}$.

First, we define $\psi_{\ast}(s)=0$ for $s\in S$; and for $M^{\otimes l}$, where $l \geq 1$, we define $\psi_{\ast}(m_{1}\otimes \dotsm \otimes m_{l})=\psi (m_{1})m_{2} \otimes \cdots \otimes m_{l}\in M^{\otimes (l-1)}$ for $m_{1},\dots,m_{l}\in M$. Finally, for $f\in \mathcal{F}_{S}(M)$ we define $\psi_{\ast} (f)\in \mathcal{F}_{S}(M)$ by setting
$\psi_{\ast} (f)(l-1)=\psi_{\ast}(f(l))$ for each integer $l>1$. Then we set
\begin{center}
$\psi_{\ast} (f)=\displaystyle \sum _{l=0}^{\infty }\psi_{\ast} (f(l)).$
\end{center}

\begin{definition} Let $\psi \in M^{*}=\mathrm{Hom}_{S}(M_{S},S_{S})$ and $f\in \mathcal{F}_{S}(M)$. We define $\delta _{\psi }:\mathcal{F}_{S}(M)\rightarrow \mathcal{F}_{S}(M)$ as
\begin{center}
$\delta _{\psi }(f)=\psi_{\ast}(\delta (f))=\displaystyle \sum _{l=0}^{\infty }\psi_{\ast}(\delta (f(l))).$
\end{center}
\end{definition}

\begin{definition} Given an $S$-bimodule $N$ we define the \emph{cyclic part} of $N$ as $N_{cyc}:=\displaystyle \sum_{j=1}^{n} e_{j}Ne_{j}$.
\end{definition}
\begin{definition} A \emph{potential}  $P$ is an element of $\mathcal{F}_{S}(M)_{cyc}$. 
\end{definition}

For each legible element $a$ of $e_{i}Me_{j}$, we let $\sigma(a)=i$ and $\tau(a)=j$.

\begin{definition} Let $P$ be a potential in $\mathcal{F}_{S}(M)$, we define a two-sided ideal $R(P)$ as the closure of the two-sided ideal of $\mathcal{F}_{S}(M)$ generated by all the elements $X_{a^{\ast}}(P)=\displaystyle \sum_{s \in L(\sigma(a))} \delta_{(sa)^{\ast}}(P)s$, $a \in T$.
\end{definition}

\begin{definition} An algebra with potential is a pair $(\mathcal{F}_{S}(M),P)$ where $P$ is a potential in $\mathcal{F}_{S}(M)$ and $M_{cyc}=0$.
\end{definition}

The following construction follows the one given in \cite[p.20]{5}. Let $k$ be an integer in $[1,n]$. Using the $S$-bimodule $M$, we define a new $S$-bimodule $\mu_{k}M=\widetilde{M}$ as:
\begin{center}
$\widetilde{M}:=\bar{e}_{k}M\bar{e}_{k} \oplus Me_{k}M \oplus (e_{k}M)^{\ast} \oplus ^{\ast}(Me_{k})$
\end{center}
where $\bar{e}_{k}=1-e_{k}$, $(e_{k}M)^{\ast}=\operatorname{Hom}_{S}((e_{k}M)_{S},S_{S})$, and $^{\ast}(Me_{k})=\operatorname{Hom}_{S}(_{S}(Me_{k}),_{S}S)$. One can show (see \cite[Lemma 8.7]{1}) that $\mu_{k}M$ is $Z$-freely generated. 

\begin{definition} Let $P$ be a potential in $\mathcal{F}_{S}(M)$ such that $e_{k}Pe_{k}=0$. Following \cite{5}, we define:
\begin{center}
$\mu_{k}P:=[P]+\displaystyle \sum_{sa \in _{k}\hat{T},bt \in \tilde{T}_{k}}[btsa]((sa)^{\ast})(^{\ast}(bt))$
\end{center}
\end{definition}
where: 
\begin{align*}
_{k}\hat{T}&=\{sa: s \in L(k),a \in T \cap e_{k}M\} \\
\tilde{T}_{k}&=\{bt: b \in T \cap Me_{k}, t \in L(k)\}.  
\end{align*}

\section{Existence of non-degenerate potentials}

For every integer $m \geq 2$, let $B(T)_{m}$ be the $F$-basis of $(M^{\otimes m})_{cyc}$ consisting of all elements of the form $x=t_{1}(x)a_{1}(x)t_{2}(x) \cdots t_{m}(x)a_{m}(x)t_{m+1}(x)$ where $t_{i}(x) \in L(\sigma(a_{i}(x)))$, $t_{m+1}(x) \in L(\tau(a_{m}(x)))$, $a_{i}(x) \in T$, and let $B(T)=\displaystyle \bigcup_{m=2}^{\infty} B(T)_{m}$. If $B(T)$ is non-empty, then $B(T)$ is clearly countable. Note than an enumeration of the elements of $B(T)$ gives rise to an algebra isomorphism between $F[Z_{x}]_{x \in B(T)}$, the free commutative $F$-algebra on the set $B(T)$, and the polynomial ring in countably many variables. 

\begin{definition} Let $P$ be a potential in $\mathcal{F}_{S}(M)$. We say that $P$ is $2$-acyclic if no element of $(M^{\otimes 2})_{cyc}$ appears in the expansion of $P$.
\end{definition}

The following definition is motivated by \cite[Definition 7.2]{5}.

\begin{definition} \label{def10} Let $k_{1}, \hdots, k_{l}$ be a finite sequence of elements of $\{1,\hdots,n\}$ such that $k_{p} \neq k_{p+1}$ for $p=1,\hdots, l-1$. We say that an algebra with potential $(\mathcal{F}_{S}(M),P)$ is $(k_{l},\hdots,k_{1})$-non-degenerate if all the iterated mutations $\bar{\mu}_{k_{1}}P$, $\bar{\mu}_{k_{2}}\bar{\mu}_{k_{1}}P, \hdots, \bar{\mu}_{k_{l}} \cdots \bar{\mu}_{k_{1}}P$ are $2$-acyclic. We say that $(\mathcal{F}_{S}(M),P)$ is non-degenerate if it is $(k_{l},\hdots, k_{1})$-non-degenerate for every sequence of integers as above.
\end{definition}

\begin{definition} If $W \in F[Z_{x}]_{x \in B(T)}$ we define $Z(W):=\{f \in F^{B(T)}: W(f(x))_{x \in B(T)} \neq 0\}$ where $F^{B(T)}$ is the $F$-vector space of all functions $B(T) \rightarrow F$.
\end{definition}

We now recall the definition of \emph{species realization} of a skew-symmetrizable integer matrix, in the sense of \cite{6} (Definition 2.22).

\begin{definition} \label{especies} Let $B=(b_{ij}) \in \mathbb{Z}^{n \times n}$ be a skew-symmetrizable matrix, and let $I=\{1,\hdots, n\}$. A species realization of $B$ is a pair $(\mathbf{S},\mathbf{M})$ such that:
\begin{enumerate}
\item $\mathbf{S}=(F_{i})_{i \in I}$ is a tuple of division rings;
\item $\mathbf{M}$ is a tuple consisting of an $F_{i}-F_{j}$ bimodule $M_{ij}$ for each pair $(i,j) \in I^{2}$ such that $b_{ij}>0$;
\item for every pair $(i,j) \in I^{2}$ such that $b_{ij}>0$, there are $F_{j}-F_{i}$-bimodule isomorphisms
\begin{center}
$\operatorname{Hom}_{F_{i}}(M_{ij},F_{i}) \cong \operatorname{Hom}_{F_{j}}(M_{ij},F_{j})$;
\end{center}
\item for every pair $(i,j) \in I^{2}$ such that $b_{ij}>0$ we have $\operatorname{dim}_{F_{i}}(M_{ij})=b_{ij}$ and $\operatorname{dim}_{F_{j}}(M_{ij})=-b_{ji}$.
\end{enumerate}
\end{definition}

\begin{theorem} \label{theo1} Suppose that the underlying field $F$ is uncountable, then $\mathcal{F}_{S}(M)$ admits a non-degenerate potential.
\end{theorem}

\begin{proof} We follow the guidelines of \cite[Corollary 7.4]{5}. Let $k_{1}, \hdots, k_{l}$ be a sequence as in Definition \ref{def10}. By \cite[Proposition 12.5]{1} there exists a potential $P \in \mathcal{F}_{S}(M)$ such that $P$ is $(k_{l},\hdots, k_{1})$-non-degenerate. Now using \cite[Proposition 12.3]{1} we can find a nonzero polynomial $W_{(k_{1},\hdots, k_{l})} \in F[Z_{x}]_{x \in B(T)}$ such that every potential belonging to $Z(W_{(k_{1},\hdots, k_{l})})$, is $(k_{l},\hdots, k_{1})$-non-degenerate. This collection $\mathcal{A}$ is a subset of $F[Z_{x}]_{x \in B(T)}$. Moreover, it is a countable family since it is indexed by a subset of all finite sequences of $\mathbb{N}$, and the latter is clearly countable. It remains to show that $\displaystyle \bigcap_{A \in \mathcal{A}} Z(A) \neq \emptyset$. We may realize the polynomial ring $F[Z_{x}]_{x \in B(T)}$ as the polynomial ring $F[t_{1},t_{2},\hdots]$. As in \cite[Corollary 7.4]{5}, since $F$ is uncountable we can find $\lambda_{1} \in F$ such that $G(\lambda_{1}) \neq 0$ for all $G \in \mathcal{A} \cap F[t_{1}]$. Then, we can find $\lambda_{2} \in F$ such that $G(\lambda_{1},\lambda_{2}) \neq 0$ for all $G \in \mathcal{A} \cap F[t_{1},t_{2}]$. Repeating this process, we can find a sequence $(\lambda_{i})_{i \geq 1}$ of elements of $F$ such that $G(\lambda_{1},\lambda_{2},\hdots) \neq 0$ for all $G \in \mathcal{A}$. Note that we can guarantee this by the fact that $\mathcal{A}$ can be written as the union of the factors that appear in the following ascending chain: 
\begin{center}
$\mathcal{A} \cap F[t_{1}] \subseteq \mathcal{A} \cap F[t_{1},t_{2}]  \subseteq \mathcal{A} \cap F[t_{1},t_{2},t_{3},\hdots] \subseteq \ldots$ 
\end{center}
This completes the proof.
\end{proof}

\begin{corollary} \label{cor} Let $B=(b_{ij}) \in \mathbb{Z}^{n \times n}$ be a skew-symmetrizable matrix with skew-symmetrizer $D=\operatorname{diag}(d_{1},\ldots,d_{n})$. If $d_{j}$ divides $b_{i,j}$ for every $j$ and every $i$, then the matrix $B$ can be realized by a species admitting a non-degenerate potential. 
\end{corollary}

\begin{proof} Note that there exists a Galois extension $U/V$ such that $\operatorname{Gal}(U/V)$ is isomorphic to $S_{m}$, where $m:=d_{1} \cdots d_{n}$. Indeed, let $F=\mathbb{C}$ denote the set of all complex numbers (or any uncountable field), $U=F(x_{1},\ldots,x_{m})$, the field of rational functions in the indeterminates $x_{1},\ldots,x_{m}$; and $V=F(s_{1},\ldots,s_{m})$, the subfield generated by the elementary symmetric polynomials $s_{1},\ldots,s_{m}$. Then $\operatorname{Gal}(U/V) \cong S_{m}$ and by Cayley's Theorem, $G:=\displaystyle \bigoplus_{i=1}^{n} \mathbb{Z}_{d_{i}}$ may be realized a subgroup of $S_{m}$. Now, let $L$ be the fixed field of $G$ in $U$, then by the Fundamental Theorem of Galois Theory we have that $U/L$ is Galois and $\operatorname{Gal}(U/L) \cong G$. Applying \cite[Proposition 11.2]{1} yields that the matrix $B$ can be realized by a species whose underlying field is precisely $L$. Note that $L$ is uncountable since it contains the uncountable field $V$. Applying Theorem \ref{theo1} yields that such species admits a non-degenerate potential. This completes the proof.
\end{proof}
\section{Rigidity and nondegeneracy}

The following definition is taken from \cite[Definition 43]{1} and it is motivated by \cite[Definition 6.7]{5}.
\begin{definition} Let $(\mathcal{F}_{S}(M),P)$ be an algebra with potential. The deformation space $\operatorname{Def}(M,P)$ is the quotient $\mathcal{F}_{S}(M)^{\geq 1}/(R(P)+[\mathcal{F}_{S}(M),\mathcal{F}_{S}(M)])$ where $[\mathcal{F}_{S}(M),\mathcal{F}_{S}(M)]$ denotes the commutator of $\mathcal{F}_{S}(M)$.
\end{definition}

The following definition is also based on \cite[Definition 6.10]{5}.
\begin{definition} An algebra with potential $(\mathcal{F}_{S}(M),P)$ is rigid if the deformation space $\operatorname{Def}(M,P)$ is zero.
\end{definition}

\begin{lemma} \label{rigidez} Let $(\mathcal{F}_{S}(M),P)$ and $(\mathcal{F}_{S}(M'),Q)$ be right-equivalent algebras with potentials. Then $(\mathcal{F}_{S}(M),P)$ is rigid if and only if $(\mathcal{F}_{S}(M'),Q)$ is rigid.
\end{lemma}

\begin{proof} Let $\varphi: \mathcal{F}_{S}(M) \rightarrow \mathcal{F}_{S}(M')$ be an algebra with $\varphi|_{S}=id_{S}$. By \cite[Theorem 5.3]{1} we have $\varphi(R(P))=R(\varphi(P))$, and by continuity of $\varphi$, it follows that $\varphi([\mathcal{F}_{S}(M),\mathcal{F}_{S}(M)])Ê\subseteq [\mathcal{F}_{S}(M'),\mathcal{F}_{S}(M')]$. This implies that $\varphi$ induces a surjection: $\psi: \mathcal{F}_{S}(M)^{\geq 1}/(R(P)+[\mathcal{F}_{S}(M),\mathcal{F}_{S}(M)]) \twoheadrightarrow \mathcal{F}_{S}(M')^{\geq 1}/(R(Q)+[\mathcal{F}_{S}(M'),\mathcal{F}_{S}(M')])$, which in fact is an isomorphism because $\varphi$ is. \end{proof}

\begin{lemma} \label{rigidez2} An algebra with potential $(\mathcal{F}_{S}(M),P)$ is rigid if and only if the mutated algebra $(\mathcal{F}_{S}(\bar{\mu}_{k}M),\bar{\mu}_{k}P)$ is rigid.
\end{lemma}

\begin{proof} This is the statement of \cite[Proposition 10.1]{1}.
\end{proof}

\begin{lemma} Let $P$ be a reduced potential in $\mathcal{F}_{S}(M)$ and let $k_{1},k_{2}$ be distinct integers in $[1,n]$ such that $\mathcal{F}_{S}(\bar{\mu}_{k_{2}}\bar{\mu}_{k_{1}}M)_{cyc}=0$. Then $P$ is a non-degenerate potential in $\mathcal{F}_{S}(M)$.
\end{lemma}

\begin{proof} Due to the fact that there are no potentials in $\mathcal{F}_{S}(\bar{\mu}_{k_{2}}\bar{\mu}_{k_{1}}M)$ it follows that $\bar{\mu}_{k_{2}}\bar{\mu}_{k_{1}}P=0$. Since mutation at $k_{2}$ preserves rigidity and mutation is an involution (see \cite[Theorem 8.12]{1}), Lemma \ref{rigidez} yields that  $(\mathcal{F}_{S}(\bar{\mu}_{k_{1}}M),\bar{\mu}_{k_{1}}P)$ is also rigid. Applying the same argument with $k_{1}$ one gets that $(\mathcal{F}_{S}(M),P)$ is a rigid and reduced algebra; therefore $P$ is non-degenerate.
\end{proof}
An induction gives the following
\begin{proposition} Let $k_{1},\ldots,k_{l}$ be a finite sequence of elements of $\{1,\ldots,n\}$ such that $k_{i} \neq k_{i+1}$ for every $i=1,\ldots,l-1$. Let $P$ be a reduced potential in $\mathcal{F}_{S}(M)$ such that $\mathcal{F}_{S}(\bar{\mu}_{k_{l}} \cdots \bar{\mu}_{k_{2}} \bar{\mu}_{k_{1}}M)_{cyc}=0$. Then $P$ is a non-degenerate potential in $\mathcal{F}_{S}(M)$.
\end{proposition}

\section{A species realization for a certain class of $4 \times 4$ skew symmetrizable matrices}
In this section we provide an example of a class of skew-symmetrizable $4 \times 4$ integer matrices, which are not globally unfoldable nor strongly primitive (in the sense of \cite[Definition 14.1]{7}), and that have a species realization admitting a non-degenerate potential. \\

In what follows, let 
\begin{equation}
B
=\begin{bmatrix}
0 & -a & 0 & b \\
1 & 0 & -1 & 0 \\
0 & a & 0 & -b \\
-1 & 0 & 1 & 0
\end{bmatrix}
\label{matrices}
\end{equation}
where $a,b$ are positive integers such that $a<b$, $a$ does not divide $b$ and $(a,b) \neq 1$. \\

Note that there are infinitely many such pairs $(a,b)$. For example, let $p$ and $q$ be primes such that $p<q$. For any $n \geq 2$, define $a=p^{n}$ and $b=p^{n-1}q$. Then $a<b$, $a$ does not divide $b$ and $(a,b)=p^{n-1} \neq 1$. Note that $B$ is skew-symmetrizable since it admits $D=\operatorname{diag}(1,a,1,b)$ as a skew-symmetrizer. \\

\begin{remark} By \cite[Example 14.4]{7} we know that the class of all matrices given by \eqref{matrices} does \emph{not} admit a global unfolding. Moreover, since we are not assuming that $a$ and $b$ are coprime, then such matrices are not strongly primitive; hence they are not covered by \cite{7}. 
\end{remark}

We have the following
\begin{proposition} \label{proposition2} 
The class of all matrices given by \eqref{matrices} are not globally unfoldable nor strongly primitive, yet they can be realized by a species admitting a non-degenerate potential.  
\end{proposition}

\begin{proof} The fact that they are not globally unfoldable follows at once by \cite[Example 14.4]{7}, and by construction, they are not strongly primitive. Let $(d_{1},d_{2},d_{3},d_{4})=(1,a,1,b)$ be the skew-symmetrizer tuple, then $d_{j}$ divides $b_{ij}$ for every $j$ and every $i$. Applying Corollary \ref{cor} yields that $B$ can be realized by a species  admitting a non-degenerate potential. This completes the proof.
\end{proof}

\section*{Acknowledgements}
I thank Raymundo Bautista and Daniel Labardini-Fragoso for helpful conversations. I also thank the referee for his valuable suggestions and remarks that improved the quality of this paper.

\end{document}